\documentclass{amsart}
\usepackage{graphicx} 
\numberwithin{equation}{section}
\usepackage{amsmath}
\usepackage{amssymb}
\usepackage{amsthm}
\usepackage{color}
\usepackage{enumerate}
\theoremstyle{definition}
\newtheorem{thm}{Theorem}[section]
\newtheorem{prop}[thm]{Proposition}
\newtheorem{rem}[thm]{Remark}
\newtheorem{lemma}[thm]{Lemma}

\newtheorem{example}[thm]{Example}
\newtheorem{cor}[thm]{Corollary}

\newcommand{\srb}{\nu_{1}}
\newcommand{\srbalpha}{\nu_{\alpha}}

\newcommand{\phisicalalpha}{\mu_{\alpha}}

\newcommand{\indusrbalpha}{\widetilde \nu_{\alpha}}
\newcommand{\density}{\rho_{1}}
\newcommand{\densityalpha}{\rho_{\alpha}}

\newcommand{\indudensityalpha}{h_{\alpha}}
\newcommand{\return}{\tau}
\newcommand{\returnalpha}{\tau_{\alpha}}

\newcommand{\lsvmapalpha}{T_{\alpha}}

\newcommand{\indumapalpha}{F_\alpha}

\newcommand{\map}{f}
\newcommand{\mapalpha}{f_{\alpha}}
\newcommand{\indudomain}{\mathcal{D}}
\newcommand{\maps}{\mathcal{S}}
\newcommand{\indulebesgue}{\tilde \lambda}
\newcommand{\uan}{u_{\alpha,n}}
\newcommand{\uanp}{u_{\alpha,n+1}}
\newcommand{\yan}{y_{\alpha,n}}
\newcommand{\yanp}{y_{\alpha,n+1}}

\title[Linear response of SRB and physical measures]{Linear response asymmetry between SRB and physical measures for families of intermittent maps with a transition point}
\author{Yuya Arima}
\date{\today}
\address{Graduate School of Mathematics, Nagoya University,
Furocho, Chikusaku, Nagoya, 464-8602, JAPAN} 
\email{yuya.arima.c0@math.nagoya-u.ac.jp}
\subjclass[2020]{37A05, 37C40, 37D25, 37E05}
\thanks{{\it Keywords}: linear response, intermittent maps, physical measure, SRB measure, Riemann zeta function, infinite ergodic theory}

\begin{document}

\begin{abstract}
    We study linear response for families of intermittent maps $\{f_\alpha\}$ whose Sinai–Ruelle–Bowen (SRB) measure undergoes a transition from finite to infinite total mass at a critical parameter value.
    Our results reveal the following fundamental asymmetry arising from this transition: 
smooth parameter dependence of the SRB measure $\srbalpha$ implies continuity of the physical measure $\phisicalalpha$ at the transition point, while simultaneously precluding its differentiability there.
In particular, although the physical measure varies continuously with respect to the parameter at the transition, it fails to admit a linear response for a large class of potentials in the usual sense.
We derive an explicit one-sided derivative formula describing this singular behavior and thereby give a quantitative characterization of how statistical properties degenerate as the total mass of the SRB measure diverges.
The key ingredient in the proof of our main theorem is a new method that relates the parameter dependence of physical measures near the transition point to the behavior of the Riemann zeta function near its pole at $1$.
\end{abstract}

\maketitle

\section{Introduction}
Let $\mathcal{M}$ be a compact Riemannian manifold and let $\{f_\alpha:\mathcal{M}\rightarrow \mathcal{M}\}_{\alpha\in J}$ be a one-parameter family of maps on $\mathcal{M}$, where $J\subset \mathbb{R}$ is an interval.
Assume that the map $\alpha\mapsto f_\alpha$ depends smoothly on $\alpha$.
From the viewpoint of physical applications in dynamical systems, it is important to investigate how the statistical properties of $f_\alpha$ depend smoothly on $\alpha$.
In order to formulate this problem mathematically, 
 we assume that for each $\alpha\in J$, the map $\mapalpha$ admits a unique physical measure $\phisicalalpha$, that is, an $\mapalpha$-invariant Borel probability measure on $\mathcal{M}$ such that the set 
\[
\left\{x\in\mathcal{M}:\lim_{n\to\infty}\frac{1}{n}\sum_{i=0}^{n-1}\phi (\mapalpha^i(x))=\int \phi d\phisicalalpha \text{ for every } \phi\in C(\mathcal{M})\right\}
\]
has positive Lebesgue measure, where $C(\mathcal{M})$ denotes the set of continuous functions on $\mathcal{M}$ (\cite{Young}).
Under this assumption, we can formulate the problem as follows: 
Given $\phi\in C(\mathcal{M})$, is the map 
\[
\alpha\mapsto R_{\text{Phy},\phi}(\alpha):= \int \phi d\phisicalalpha
\]
differentiable? If so, can one derive an explicit formula for its derivative in terms of $\mapalpha$, $\phisicalalpha$ and $\phi$? 
This line of research is known as linear response. 

One of the pioneering works on linear response is due to Ruelle \cite{Ruelle1997} for uniformly hyperbolic Axiom A dynamical systems. Following Ruelle’s work, the theory of linear response has been further developed by many researchers in various settings. 
For instance, we refer to \cite{Baladiunimodalpicewise, Baladiwhitney, Ruelle2009} for results on unimodal maps, \cite{BahsounSaussol, VivianeTodd, Korepanov, Leppanen} for intermittent maps, \cite{Dolgopyat, Zhang} for partially hyperbolic systems, and \cite{Bahsoun2020, HarryNakano, DavorTokman, DragivcevicSedro, GalatoloSedro} for random dynamical systems, as well as the references therein.
Furthermore, linear response has found applications beyond mathematics (see, for example, \cite{Lucarini, ragone2016new}).

Intermittent maps have attracted attention from both mathematicians and physicists. 
For results on intermittent maps, see, for example, \cite{Gouezel, LSB, Melbourne, PesinSenti, PollicottWeiss, PomeauManneville, Sarig2002subexponential, Terhasiu, Thaler1980, ThalerZweimuller} and the references therein. 
For a family $\{\mapalpha\}_{\alpha\in J}$ of intermittent maps on $\mathcal{M}$, the following phenomenon often occurs: There exists $\alpha^*$ in the interior of $J$ such that for each $\alpha<\alpha^*$, $\phisicalalpha$ is the unique physical measure that is absolutely continuous with respect to the Lebesgue measure $\lambda$ on $\mathcal{M}$, and for each $\alpha\geq \alpha^*$, $\phisicalalpha$ is the Dirac measure $\delta_p$ at some point $p \in \mathcal{M}$.
Moreover, for each $\alpha\geq \alpha^*$ there exists an infinite $\mapalpha$-invariant ergodic measure $\tilde\nu_\alpha$ that is absolutely continuous with respect to $\lambda$. We call $\alpha^*$ the transition point of the family $\{\mapalpha\}_{\alpha\in J}$.
Suppose that we are in this situation.
We also assume that there exists a Borel set $\indudomain\subset \mathcal{M}$ such that $0<\phisicalalpha(\indudomain)$ for all $\alpha<\alpha^*$ and $0<\tilde\nu_\alpha(\indudomain)<\infty$ for all $\alpha^*\leq\alpha$. We define $\srbalpha:=(\phisicalalpha(\indudomain))^{-1}\phisicalalpha$  for all $\alpha<\alpha^*$ and $\srbalpha:=\tilde\nu_\alpha(\indudomain)^{-1}\tilde \nu_{\alpha}$ for all $\alpha^*\leq\alpha$. 
In this paper, following the spirit of Young \cite{Young}, we refer to $\srbalpha$ as the SRB measure for $\mapalpha$. 

From the viewpoint of infinite ergodic theory (see, for example, \cite{Aaronsonbook, AaronsonDenker, Gouezel, Melbourne, Sera, ThalerZweimuller} and the references therein), another way to formulate the above problem mathematically is to consider the differentiability of the following map: For $\phi \in C(\mathcal{M})$, we consider the map
\begin{align*}
    \alpha\mapsto R_{\text{SRB},\phi}(\alpha):=\int \phi d\srbalpha.
\end{align*}
Indeed, this direction of research was considered by Bahsoun and Saussol \cite{BahsounSaussol}. 
Their work suggests that the SRB measure $\srbalpha$ depends smoothly on the parameter $\alpha$ even at the transition point, and that the map $R_{\text{SRB},\phi}$ does as well.

In view of the above motivation, it is a fundamental problem to investigate how the smooth dependence of the SRB measure $\srbalpha$ on the parameter $\alpha$ influences that of the physical measure $\phisicalalpha$.
However, this problem does not appear to have been addressed in the existing literature.
In this paper, we provide a first systematic study of this problem. 
Since, for $\alpha < \alpha^*$, $\srbalpha$ and $\phisicalalpha$ coincide up to a constant multiple and $\phisicalalpha=\delta_p$ for each $\alpha\geq \alpha^*$, our main focus is on the transition point $\alpha^*$.

Our main theorem (Theorem \ref{thm main super abstract}) shows that
under the assumption that the smooth dependence of the SRB measure $\srbalpha$ on the parameter $\alpha$ 
implies that, for any $\phi\in C(\mathcal{M})$, the map $R_{\text{Phy},\phi}$ admits a one-sided derivative at $\alpha^*$ given by an explicit and simple formula.
This result provides a necessary and sufficient condition on $\phi$ for the differentiability of $R_{\text{Phy},\phi}$ at $\alpha^*$. 
From this observation, we obtain the following two results: 
First, we provide an example of a family $\{\mapalpha\}_{\alpha \in J}$ of intermittent maps for which $R_{\text{SRB},\psi}$ is differentiable at $\alpha^*$ for all H\"older continuous potential $\psi\in C(\mathcal{M})$  with $\int |\psi|  d\srbalpha < \infty$, whereas $R_{\text{Phy},\phi}$ is not differentiable at $\alpha^*$ for a large class of H\"older potentials (see Theorem \ref{thm main pomeau}).   
Second, under the assumptions of the main theorem, $R_{\text{Phy},\phi}$ fails to be differentiable at $\alpha^*$ for a large class of continuous potentials (see Remark \ref{rem compare srb and physical}). 
These results reveal a fundamental asymmetry arising from the transition of the SRB measure from finite to infinite total mass: 
smooth parameter dependence of the SRB measure $\srbalpha$ implies continuity of the physical measure $\phisicalalpha$ at the transition point, while simultaneously precluding its differentiability.

\subsection{Precise statement of the main theorem in a general framework}\label{sec abstruct setting}
In this section, we formulate our main results in a general framework. 

Let $J\subset \mathbb{R}$ be a non-trivial closed interval and let $\{X_\alpha\}_{\alpha\in J}$ be a family of metric spaces. 
For each $\alpha\in J$, let $m_\alpha$ be a reference Borel probability measure on $X_\alpha$. 
Unless otherwise specified, we fix the family $\{(X_\alpha,m_\alpha)\}_{\alpha\in J}$ throughout this section.

For each $\alpha\in J$, let $\mapalpha:X_\alpha\rightarrow X_\alpha$ be a Borel measurable map. 
We consider the following conditions:
\begin{itemize}
    \item[(X1)]
    For each $\alpha\in J$ there exists a $\sigma$-finite Borel measure $\srbalpha$ on $X_\alpha$ that is absolutely continuous with respect to the reference probability measure $m_\alpha$ and is conservative, ergodic and invariant with respect to $f_\alpha$.
    Moreover, for each $\alpha\in J$ there exists a Borel set $\indudomain_\alpha$ such that we have $\srbalpha(\indudomain_\alpha)=1$.
\item[(X2)] For each $\alpha\in J$ there exists $p_\alpha\in X_\alpha$ such that for any open neighborhood $O\subset X_\alpha$ of $p_\alpha$ we have $\srbalpha(X_\alpha\setminus O)<\infty$. 
\end{itemize}
We refer to $\srbalpha$ as a SRB measure for $(X_\alpha,\mapalpha)$.

Let $\alpha\in J$. Suppose that $\{\mapalpha\}_{\alpha\in J}$ satisfies (X1).
We define the first return time function $\returnalpha:\indudomain_\alpha\rightarrow \mathbb{N}\cup\{\infty\}$ of $f_{\alpha}$ by 
\[
\returnalpha(x):=\inf\{n\in\mathbb{N}:f_\alpha^n(x)\in\indudomain_\alpha\}.
\]
and the first return map $\indumapalpha:\{\returnalpha<\infty\}\rightarrow \indudomain_\alpha$ by 
\[
\indumapalpha(x):=f_\alpha^{\returnalpha(x)}(x).
\]
Since $\srbalpha(\indudomain_\alpha)=1>0$, and $\srbalpha$ is conservative and $\mapalpha$-invariant, Poincar\'e recurrence theorem implies that $\srbalpha(\{\returnalpha=\infty\})=0$. We also consider the following conditions: 
\begin{itemize}
\item[(X3)] There exist functions $\alpha\in J\mapsto v(\alpha)\in (0,\infty)$, $\alpha\in J\mapsto u(\alpha)\in [0,\infty)$, $\alpha^*\in J$ and a constant $D\geq 1$ such that $v(\alpha^*)=1$, $u(\alpha^*)=0$ $v(\alpha)-u(\alpha)>1$ if $\alpha<\alpha^*$, $v(\alpha)+u(\alpha)<1$ if $\alpha>\alpha^*$, 
\begin{align}\label{eq X3}
\lim_{\alpha\to \alpha^*-0}\frac{v(\alpha)-v(\alpha^*)}{\alpha-\alpha^*}=v'_-(\alpha^*) \in \mathbb{R}
,\ \lim_{\alpha\to \alpha^*}\frac{u(\alpha)}{\alpha-\alpha^*}=0
\end{align}
and 
for all $\alpha\in J$ and $n\in\mathbb{N}$ we have
\begin{align}\label{eq comparability}
    {D}^{-1}{n^{-(v(\alpha)+u(\alpha))}}\leq {\srbalpha(\{\returnalpha>n\})} 
    \leq D{n^{-(v(\alpha)-u(\alpha))}}
\end{align}
\item[(X3')] There exist a continuous function $\alpha\in J\mapsto c_\alpha\in (0,\infty)$, a function $\alpha\in J\mapsto v(\alpha)\in (0,\infty)$, $\alpha^*\in J$, a constant $D>0$ and $\epsilon>0$ such that we have $v(\alpha^*)=1$, 
 $v(\alpha)>1$ if $\alpha<\alpha^*$, $v(\alpha)<1$ if $\alpha>\alpha^*$, 
 \begin{align}\label{eq X3'}
 \lim_{\alpha\to \alpha^*-0}\frac{v(\alpha)-v(\alpha^*)}{\alpha-\alpha^*}=v'_-(\alpha^*)\in \mathbb{R}
 \end{align}
 and 
for all $\alpha\in J$ and $n\in\mathbb{N}$ we have
\begin{align}\label{eq expansion of tail abstruct}
    |\srbalpha(\{\returnalpha>n\})-c_\alpha n^{-v(\alpha)}|\leq Dn^{-v(\alpha)-\epsilon}.
\end{align}
\end{itemize}
Note that (X3') implies (X3) with $u\equiv0$. For further discussion of the assumptions (X3) and (X3') see the final part of this section.

\begin{example}\label{example MP map}
    Let $J\subset \mathbb{R}$ be a non-trivial closed interval containing $1$ and let $\alpha\in J$.
    We set $X_\alpha:=[0,1]$ and $m_\alpha:=\lambda$, where $\lambda$ denotes the Lebesgue measure on $[0,1]$.
    We define the map $\lsvmapalpha:I\rightarrow{I}$ (\cite{LSB}) by 
\begin{align}\label{eq def lsv map intro} 
\lsvmapalpha(x):=
 \left\{
 \begin{array}{cc}
 x+2^{\alpha}x^{1+\alpha}  & \text{if}\  x\in[0,1/2]\\
   2x-1
   & \text{if}\ x\in(1/2,1]
 \end{array}
 \right. .
\end{align}
In Section \ref{sec proof sec 3} we show that
 $\{\lsvmapalpha\}_{\alpha\in J}$ satisfies (X1), (X2) and (X3') with $\indudomain_\alpha=[1/2,1]$, $p_\alpha=0$, $\alpha^*=1$, 
\[ 
c_\alpha=\frac{\densityalpha(1/2)}{4(\alpha 2^\alpha)^{1/\alpha}}
\text{ and }v(\alpha)=\frac{1}{\alpha} \text{ for each }\alpha\in J,
\]
where $\densityalpha$ denotes the Radon-Nikodym derivative of $\srbalpha$ with respect to $\lambda$, chosen to be continuous on $(0,1]$ (see \cite{Thaler1983, Thaler1995}).
\end{example}

Let $\{\mapalpha\}_{\alpha\in J}$ satisfy (X1), (X2) and (X3) and let $\alpha\in J$. 
Notice that  
\begin{align}\label{eq cal return}
    &\int \returnalpha d\srbalpha=\sum_{l=1}^{\infty}l\srbalpha(\{\returnalpha=l\})=\sum_{l=1}^\infty \sum_{k=0}^{l-1}\srbalpha(\{\returnalpha=l\})
    \\&=\sum_{k=0}^\infty\sum_{l=k+1}^\infty \srbalpha(\{\returnalpha=l\})=\sum_{k=0}^\infty\srbalpha(\{\returnalpha>k\}).\nonumber
\end{align}
and 
\begin{align}\label{eq lower and upper bound for sum of tail}
D^{-1}\sum_{k=1}^\infty \frac{1}{k^{v(\alpha)+u(\alpha)}} +1
\leq
    \sum_{k=0}^\infty\srbalpha(\{\returnalpha>k\})
    \leq 
    D \sum_{k=1}^\infty \frac{1}{k^{v(\alpha)-u(\alpha)}} +1
\end{align}
by \eqref{eq comparability}. Therefore, since $v(\alpha)+u(\alpha)\leq1$ if $\alpha\geq\alpha^*$ and $v(\alpha)-u(\alpha)>1$ if $\alpha<\alpha^*$, we obtain
\[
\int \returnalpha d\srbalpha=\infty \text{ if }\alpha\geq\alpha^* \text{ and }
\int \returnalpha d\srbalpha<\infty \text{ if }\alpha<\alpha^*.
\]
Note that since $\returnalpha$ is the first return time function, Kac's formula implies that 
\begin{align}\label{eq kac}
    \srbalpha(X_\alpha)=\int \returnalpha d\srbalpha
\end{align}
(see, for example, \cite[Proposition 1.4.3 and Corollary 1.4.4]{viana}). Hence,  $\srbalpha$ is finite if $\alpha<\alpha^*$ and $\srbalpha$ is infinite if $\alpha\geq \alpha^*$.  

We define the Borel probability measure $\phisicalalpha$ on $X_\alpha$ by 
\begin{align}\label{eq def phisical super abstruct} 
\phisicalalpha:=
 \left\{
 \begin{array}{cc}
 ({\srbalpha(X_\alpha)})^{-1}{\srbalpha}  & \text{if}\  \alpha<\alpha^* \text{ and }\alpha\in J\\
   \delta_{p_\alpha}   & \text{if}\ \alpha\geq\alpha^* \text{ and }\alpha\in J.
 \end{array}
 \right. .
\end{align}
The proof of the following lemma is given in Section \ref{sec proof sec 2}: We denote by $C_{b}(X)$ the set of bounded continuous functions on a metric space $X$. 
\begin{lemma}\label{lemma motivation physical}
Let $X$ be a metric space and let $m$ be a Borel probability measure on $X$.
    Let $f:X\rightarrow X$ be a measurable map and let $\nu$ be a $\sigma$-finite infinite Borel measure on $X$. Suppose that $\nu$ is invariant ergodic and conservative with respect to $f$ and absolutely continuous with respect to $m$, and there exists $p\in X$ such that for all open neighborhood $O\subset X$ of $p$ we have $\srbalpha(X\setminus O)<\infty$. 
    Then, the set 
    \[
    \left\{x\in X: \lim_{n\to\infty}\frac{1}{n}\sum_{i=0}^{n-1}\phi(f^i(x))=\phi(p) \text{ for all } \phi\in C_b(X)\right\}
    \]
    has positive measure with respect to $m$.
\end{lemma}
Since $\srbalpha$ is absolutely continuous with respect to the reference probability measure $m_\alpha$, 
Birkhoff's ergodic theorem and Lemma \ref{lemma motivation physical} imply that for each $\phi \in C_b(X_\alpha)$ the set  
$
  B_\alpha(\phi):=  \left\{x\in X_\alpha: \lim_{n\to\infty}\frac{1}{n}\sum_{i=0}^{n-1}\phi(\mapalpha^i(x))=\int \phi d\phisicalalpha \right\}
    $
    has positive measure with respect to $m_\alpha$. 
Moreover, by \cite[Remark 3.1.16]{PrzytyckiUr}, if $X_\alpha$ is a compact metric space then $\bigcap_{\phi\in C(X_\alpha)} B_\alpha(\phi)$
    has positive measure with respect to $m_\alpha$, where $C(X_\alpha)$ denotes the set of continuous functions on $X_\alpha$.
    Motivated by this observation, we refer to $\phisicalalpha$ as a physical measure for $(X_\alpha,\mapalpha)$. We are now in the position to state our main theorem.

\begin{thm}\label{thm main super abstract}
We assume that a family $\{\mapalpha:X_\alpha\rightarrow X_\alpha\}_{\alpha\in J}$ of Borel measurable maps satisfies (X1), (X2) and (X3). Let $\Phi:=\{\phi_\alpha:X_\alpha\rightarrow \mathbb{R}\}_{\alpha\in J}$ be a family of measurable potentials such that for all $\alpha\in (\inf J,\alpha^*]$ we have $\int|\phi_\alpha| d \phisicalalpha<\infty$, $\int |\phi_{\alpha^*}-\phi_{\alpha^*}(p_{\alpha^*})| d\nu_{\alpha^*}<\infty$ and 
\begin{align}\label{eq continuity main theorem}
\lim_{\alpha\to\alpha^*-0}\int (\phi_{\alpha}-\phi_{\alpha^*}(p_{\alpha^*}))d\srbalpha=\int (\phi_{\alpha^*}-\phi_{\alpha^*}(p_{\alpha^*}))d\nu_{\alpha^*}.
\end{align}
Then, we have
\begin{align*}
&D^{-1}v_-'(\alpha^*){\int (\phi_{\alpha^*}-\phi_{\alpha^*}(p_{\alpha^*}))d\nu_{\alpha^*}}
\leq \liminf_{\alpha\to\alpha^*-0}
\frac{\int \phi_\alpha d\phisicalalpha-\int \phi_{\alpha^*} d\mu_{\alpha^*}}{\alpha-\alpha^*}
\\&\leq\limsup_{\alpha\to\alpha^*-0}\frac{\int \phi_\alpha d\phisicalalpha-\int \phi_{\alpha^*} d\mu_{\alpha^*}}{\alpha-\alpha^*}
\leq 
D
v_-'(\alpha^*){\int (\phi_{\alpha^*}-\phi_{\alpha^*}(p_{\alpha^*}))d\nu_{\alpha^*}}, 
\end{align*}
where $D$ is the constant appearing in \eqref{eq comparability}.
Moreover, if $\{\mapalpha\}_{\alpha\in J}$ satisfies (X3'), we obtain 
\begin{align}\label{eq abstract main eq}
\lim_{\alpha\to\alpha^*-0}\frac{\int \phi_\alpha d\phisicalalpha-\int \phi_{\alpha^*} d\mu_{\alpha^*}}{\alpha-\alpha^*}
=
\frac{v_-'(\alpha^*)\int (\phi_{\alpha^*}-\phi_{\alpha^*}(p_{\alpha^*}))d\nu_{\alpha^*}}{c_{\alpha^*}}. 
\end{align}
\end{thm}

\begin{rem}
If the family of metric spaces $\{X_\alpha\}_{\alpha\in J}$ does not depend on $\alpha$, then family $\Phi:=\{\phi_\alpha:X_\alpha\rightarrow \mathbb{R}\}_{\alpha\in J}$ of measurable functions in the above theorem can be taken to be a single function $\phi$. In this case, \eqref{eq continuity main theorem} is replaced by 
\[
\lim_{\alpha\to\alpha^*-0}\int \psi d\srbalpha=\int \psi d\nu_{\alpha^*}, \text{ where }\psi:=\phi-\phi(p_{\alpha^*}).
\]

\end{rem}

We denote by $\text{Int}(J)$ the interior of $J\subset \mathbb{R}$.

\begin{rem}\label{rem compare srb and physical}
Assume that we are in the same setting as in Theorem \ref{thm main super abstract} and $\alpha^*\in \text{Int}(J)$. 
If $p_\alpha$ is independent of $\alpha$ then the map
$\alpha\mapsto R_{\text{Phy},\Phi}(\alpha):=\int \phi_\alpha d\phisicalalpha$ is differentiable on $(\alpha^*,\sup J)$ and its derivative is identically zero. 
In particular, by Theorem \ref{thm main super abstract}, $R_{\text{Phy},\Phi}$ is differentiable at $\alpha^*$ if and only if 
\begin{align}\label{eq main n s condition}
v_-'(\alpha^*)\int (\phi_{\alpha^*}-\phi_{\alpha^*}(p_{\alpha^*}))d\nu_{\alpha^*}=0
.    
\end{align}
On the other hand, assume that $\{f_\alpha\}_{\alpha\in J}$ satisfies (X3') and that the limit 
\[
g(\alpha^*):=
\lim_{\alpha\to \alpha^*+0}\frac{\phi_\alpha(p_\alpha)-\phi_{\alpha^*}(p_{\alpha^*})}{p_\alpha-p_{\alpha^*}} \in[-\infty,\infty]
\]
exists (if $p_\alpha$ does not depend on $\alpha\in J$ then we set $g(\alpha^*):=0$).
Then $R_{\text{Phy},\Phi}$ is differentiable at $\alpha^*$ if and only if 
\begin{align}\label{eq main n s condition strong}
\frac{v_-'(\alpha^*)\int (\phi_{\alpha^*}-\phi_{\alpha^*}(p_{\alpha^*}))d\nu_{\alpha^*}}{c_{\alpha^*}}=g(\alpha^*).   
\end{align}
Note that it is impossible to normalize $\phi_{\alpha^*}$ so that it satisfies \eqref{eq main n s condition}.
Indeed, for a continuous potential $\phi_\alpha$ with $\int |\phi_{\alpha^*}-\phi_{\alpha^*}(p_{\alpha^*})| d\nu_{\alpha^*}<\infty$ and $c\in\mathbb{R}$ if we consider $\tilde \phi_{\alpha^*}:=\phi_{\alpha^*}-c$ then we have
\[
\int(\tilde\phi_{\alpha^*}-\tilde\phi_{\alpha^*}(p_{\alpha^*})) d\nu_{\alpha^*}=\int(\phi_{\alpha^*}-\phi_{\alpha^*}(p_{\alpha^*})) d\nu_{\alpha^*}.
\] 
\end{rem}

Given a family of measurable functions $\{\psi_n\}_{n\in\mathbb{N}}$ on a  metric space $X$ with a reference Borel probability measure $m$, we define
\[
\operatorname*{ess\,limsup}_{n\to\infty} \psi_n
:=
\inf \left\{
c\in\mathbb{R}
:
m\!\left(
\left\{ x\in X :
\limsup_{n\to\infty}\psi_n(x)>c
\right\}
\right)=0
\right\},
\]
and similarly for $\operatorname*{ess\,liminf}_{n\to\infty} \psi_n$.
If
\[
\operatorname*{ess\,liminf}_{n\to\infty} \psi_n
=
\operatorname*{ess\,limsup}_{n\to\infty} \psi_n,
\]
then we write
\[
\operatorname*{ess\,lim}_{n\to\infty} \psi_n
:=
\operatorname*{ess\,limsup}_{n\to\infty} \psi_n
=\operatorname*{ess\,liminf}_{n\to\infty} \psi_n.
\]

\begin{cor}\label{cor time average}
We assume that a family $\{\mapalpha:X_\alpha\rightarrow X_\alpha\}_{\alpha\in J}$ of Borel measurable maps satisfies (X1), (X2) and (X3').
We also assume that for all $\alpha\in (\inf J,\alpha^*)$ the probability measure $\phisicalalpha$ is equivalent to $m_\alpha$. 
Let $\{\phi_\alpha:X_\alpha\rightarrow \mathbb{R}\}_{\alpha\in J}$ be a family of measurable potentials such that for all $\alpha\in (\inf J,\alpha^*]$ we have $\int|\phi_\alpha| d \phisicalalpha<\infty$, $\int |\phi_{\alpha^*}-\phi_{\alpha^*}(p_{\alpha^*})| d\nu_{\alpha^*}<\infty$ and 
\begin{align}\label{eq continity corollary}
\lim_{\alpha\to\alpha^*-0}\int (\phi_{\alpha}-\phi_{\alpha^*}(p_{\alpha^*}))d\srbalpha=\int (\phi_{\alpha^*}-\phi_{\alpha^*}(p_{\alpha^*}))d\nu_{\alpha^*}.
\end{align}
Then, we have
\[
\lim_{\alpha\to\alpha^*-0}
\operatorname*{ess\,lim}_{n\to\infty}
\left|
\frac{\frac{1}{n}\sum_{i=0}^{n-1} \phi_\alpha\circ\mapalpha^i-\phi_{\alpha^*}(p_{\alpha^*})}{\alpha-\alpha^*}-\frac{v_-'(\alpha^*)\int (\phi_{\alpha^*}-\phi_{\alpha^*}(p_{\alpha^*}))d\nu_{\alpha^*}}{c_{\alpha^*}}\right|=0.
\]
\end{cor}
See Theorem \ref{thm main pomeau} for the results corresponding to 
Theorem \ref{thm main super abstract}, Remark \ref{rem compare srb and physical}, 
and Corollary \ref{cor time average} applied to Example \ref{example MP map}.

Let $\{f_\alpha\}_{\alpha\in J}$ satisfy (X1) and (X2).
In the setting of intermittent maps, the transition of the SRB measure from finite to infinite total mass leads to substantial changes in the statistical properties of the system.
We mention two representative phenomena.

First, when $\alpha<\alpha^*$, Birkhoff’s ergodic theorem yields a strong law of large numbers for integrable potentials with respect to $\phisicalalpha$. In contrast, for $\alpha\ge \alpha^*$ no normalizing sequence gives rise to a non-trivial almost-sure limit for non-negative integrable potentials (see \cite{Aaronson1981} and \cite[Theorem 2.4.1]{Aaronsonbook}).

Second, the thermodynamic formalism changes at the transition point (see \cite{kessebohmer2004multifractaltobeappdated}):
For simplicity, we assume that we are in the setting of Example \ref{example MP map}. For $\alpha\in J$ and $t\in \mathbb{R}$ we denote by $P_\alpha(t)$ the topological pressure for $-t\log |\mapalpha'|$ (see \cite[Chapter 9]{walters2000introduction} for details of the topological pressure). 
For $\alpha<\alpha^*$ the pressure function $t\mapsto P_\alpha(t)$ exhibits a first-order phase transition at $t=1$, whereas for $\alpha\ge\alpha^*$ this non-differentiability disappears.

It is therefore natural to seek a quantitative description of these statistical changes at the transition point.
Nevertheless, the work of Bahsoun and Saussol \cite{BahsounSaussol} indicates that for broad classes of smooth families $\{f_\alpha\}_{\alpha\in J}$ the SRB measure $\srbalpha$ depends smoothly on the parameter $\alpha$ at the transition point $\alpha^*$. 
Hence, these statistical changes at the transition point are not reflected in linear response of the SRB measure.

Motivated by this observation, we investigate linear response of the physical measure under the assumption that the SRB measure $\srbalpha$ depends smoothly on the parameter $\alpha$ at the transition point $\alpha^*$. More precisely, we assume that the transition of the SRB measure from finite to infinite total mass occurs smoothly. 
This assumption is formulated through conditions (X3) and (X3'). 
The smoothness is reflected in \eqref{eq X3} and \eqref{eq X3'}.

Conditions \eqref{eq comparability} and \eqref{eq expansion of tail abstruct} are standard in infinite ergodic theory (see, for example, \cite{BruinTerhesiu, Melbourne, Terhesiu, Terhasiu} and \cite{Sarig2002subexponential, Sera, ThalerZweimuller}, and the references therein).
Accordingly, techniques for verifying these types of hypotheses have been developed in several settings beyond Pomeau–Manneville maps, including Kleinian groups (see Stadlbauer and Stratmann \cite[Section 2]{StadlbauerStratmann}) and parabolic rational maps (see Aaronson, Denker and Urba\'{n}ski \cite[Sections 8 and 9]{AaronsonDenker}).
Therefore, Extending the arguments of Bahsoun and Saussol \cite{BahsounSaussol} to these contexts, one expects that Theorem \ref{thm main super abstract} applies there as well.

However, the aim of this paper is to provide a systematic study of 
how the smooth dependence of the SRB measure $\srbalpha$ on the parameter $\alpha$ influences that of the physical measure $\phisicalalpha$. We also provide a simple example of a family $\{\mapalpha\}_{\alpha \in J}$ of intermittent maps 
for which linear response holds for the SRB measure at the transition point but fails for the physical measure.
For this reason, we do not go into details in the setting of Kleinian groups and parabolic rational maps.

As mentioned above, Theorem \ref{thm main super abstract} and Remark \ref{rem compare srb and physical} reveal a fundamental asymmetry arising from the transition of the SRB measure from finite to infinite total mass: 
for a large class of families $\Phi$ of potentials, smooth parameter dependence (X3) or (X3') of the SRB measure $\srbalpha$ implies continuity of $R_{\text{Phy},\Phi}$ at $\alpha^*$, while simultaneously precluding its differentiability. Moreover, condition (X3') provides a quantitative description of this singular behavior.

The main difference between our results and the results of Bahsoun and Saussol \cite{BahsounSaussol}, Baladi and Todd \cite{VivianeTodd} and Korepanov \cite{Korepanov} for families of intermittent maps is as follows:

In \cite{BahsounSaussol, VivianeTodd, Korepanov}, the phase space is assumed to be a closed interval $I\subset\mathbb{R}$ equipped with the normalized Lebesgue measure. In contrast, our setting does not require the phase space to be compact or one-dimensional. Moreover, those works assume that each map $f_\alpha$ admits a Markov partition, whereas we impose no such assumption.

\cite{VivianeTodd, Korepanov} studied linear response for the physical measure away from the transition point. By contrast, our primary focus is the behavior of the physical measure at the transition point, where extending their methods is not straightforward.
\cite{BahsounSaussol} treat the transition point, but restrict attention to the SRB measure.

Our main difficulty is that $\alpha^*$ is the transition point.
For this reason, it is challenging to apply the cone techniques developed by Baladi and Todd \cite{VivianeTodd}, which have now become the standard approach for the study of linear response in families of intermittent maps.
A key insight for overcoming this difficulty is that, as can already be seen from \eqref{eq cal return} and \eqref{eq lower and upper bound for sum of tail}, the asymptotic behavior of 
\begin{align}\label{eq intro asy}
\frac{\int \phi_\alpha d\phisicalalpha-\int \phi_{\alpha^*} d\mu_{\alpha^*}}{\alpha-\alpha^*},
\text{ as }\alpha\to\alpha^*-0,    
\end{align}
is closely related to the Riemann zeta function $\zeta$. Moreover, a crucial fact is that the Riemann zeta function has a simple pole at $1$. These simple observations enable us to derive the asymptotic behavior of \eqref{eq intro asy} in a rather general setting.

The outline of this paper is as follows. In Section \ref{sec proof sec 2}, we give the proofs of the results in Section \ref{sec abstruct setting}. In Section \ref{sec example}, we introduce a class of families of intermittent maps on $[0,1]$ for which the results of Section \ref{sec abstruct setting} hold for arbitrary H\"older continuous potentials. 
In particular, this class includes Example \ref{example MP map}. In Section \ref{sec proof sec 3}, we verify that any family of intermittent maps belonging to the class introduced in Section \ref{sec example} satisfies the assumptions of Theorem \ref{thm main super abstract} and \ref{cor time average} for arbitrary H\"older continuous potentials.
In Section \ref{sec PMmap}, we show that Example \ref{example MP map} is an example of a family of intermittent maps for which linear response holds for the SRB measure at the transition point, but fails for the physical measure.

\section{Proofs of the results in Section \ref{sec abstruct setting}}\label{sec proof sec 2}
In this section, we give proofs of the results in Section \ref{sec abstruct setting}.
In the following,
we use the following uniform version of Landau's notation: For a interval $J\subset \mathbb{R}$, $\{\gamma_{n,\alpha}\}_{n\in\mathbb{N},\alpha\in J}\subset\mathbb{R}$ and $\{x_{n,\alpha}\}_{n\in\mathbb{N},\alpha\in J}\subset\mathbb{R}$ we write
$x_{n,\alpha}= O(\gamma_{n,\alpha})$ if there exists a constant $C>0$ such that for all $n\in\mathbb{N}$ and $\alpha\in J$ we have 
\[
|x_{n,\alpha}|\leq C|\gamma_{n,\alpha}|.
\]

\emph{Proof of Lemma \ref{lemma motivation physical}.}
Let $X$ be a metric space and let $m$ be a Borel probability measure on $X$. Let $f:X\rightarrow X$ be a measurable map and let $\nu$ be a $\sigma$-finite infinite Borel measure on $X$. Suppose that $\nu$ is invariant ergodic and conservative with respect to $f$ and absolutely continuous with respect to $m$ and there exists $p\in X$ such that for all open neighborhood $O\subset X$ of $p$ we have $\srbalpha(X\setminus O)<\infty$.

By Birkhoff's ergodic theorem (see \cite[Exercise 2.2.1]{Aaronsonbook}), for each $n\in\mathbb{N}$ there exists a Borel set $\Omega_n\subset X$ such that $\nu(X\setminus \Omega_n)=0$ and for all $x\in \Omega_n$ we have
\begin{align}\label{eq proof of motivation 0 limit}
    \lim_{n\to\infty}\frac{1}{n}\sum_{i=0}^{n-1} 1_{X\setminus B(p,1/n)} (f^{i}(x))=0,
\end{align}
where $B(p,1/n)$ is the open ball centered at $p$ with radius $1/n$ with respect to the metric on $X$ and $1_{A}$ denotes the indicator function of the Borel set $A\subset X$. We set $\Omega=\bigcap_{n\in\mathbb{N}} \Omega_n$. Note that, since $\nu(X\setminus \Omega)=\nu(\bigcup_{n\in\mathbb
N}X\setminus \Omega_n)=0$, we have $\nu(\Omega)>0$. Since $\nu$ is absolutely continuous with respect to $m$, it is enough to show that 
\begin{align}\label{eq proof of motivation enough}
    \Omega\subset 
    \left\{x\in X: \lim_{n\to\infty}\frac{1}{n}\sum_{i=0}^{n-1}\phi(f^i(x))=\phi(p) \text{ for all } \phi\in C_b(X)\right\}. 
\end{align}
Let $x\in \Omega$ and let $\phi\in C_b(X)$. We fix an arbitrary $\varepsilon>0$. Since $\phi$ is continuous at $p$, there exists $M\geq 1$ such that for all $x\in B(p,1/M)$ we have 
\begin{align}\label{eq cont phi p}
    |\phi(x)-\phi(p)|<\varepsilon.
\end{align}
Notice that 
\begin{align}\label{eq clerly eq}
    1_X=1_{X\setminus B(p,1/M)}+1_{B(p,1/M)}.
\end{align}
Combining this with \eqref{eq proof of motivation 0 limit}, we obtain 
$
\lim_{n\to\infty} \frac{1}{n}\sum_{i=0}^{n-1}1_{B(p,1/M)}(f^i(x))=1.
$
Therefore, there exists $N\geq 1$ such that for all $n\geq N$ we have 
\[
\left|\frac{1}{n}\sum_{i=0}^{n-1}1_{B(p,1/M)}(f^i(x))-1\right|<\varepsilon
\text{ and }
\left|\frac{1}{n}\sum_{i=0}^{n-1}1_{X\setminus B(p,1/M)}(f^i(x))\right|<\varepsilon.
\]
Hence, by \eqref{eq cont phi p} and \eqref{eq clerly eq}, for all $n\geq N$ we obtain 
\begin{align*}
    &\left|\frac{1}{n}\sum_{i=0}^{n-1} \phi(f^i(x))-\phi(p)\right|
    \leq\left|\frac{1}{n}\sum_{i=0}^{n-1} (\phi-\phi(p)) \cdot 1_{B(p,1/M)}(f^i(x))\right| 
    \\&+2\max\{1,\|\phi\|_\infty\}\left|\frac{1}{n}\sum_{i=0}^{n-1}1_{X\setminus B(p,1/M)}(f^i(x))\right|<\epsilon(1+\epsilon)+2\max\{1,\|\phi\|_\infty\}\epsilon.
    \end{align*}
This implies \eqref{eq proof of motivation enough} and the proof is complete.\qed

Let $J\subset \mathbb{R}$ be a non-trivial closed interval and let $\{X_\alpha\}_{\alpha\in J}$ be a family of metric spaces. 
For each $\alpha\in J$, let $m_\alpha$ be a reference Borel probability measure on $X_\alpha$.

\emph{Proof of Theorem \ref{thm main super abstract}.}
We assume that a family $\{\mapalpha:X_\alpha\rightarrow X_\alpha\}_{\alpha\in J}$ of Borel measurable maps satisfies (X1), (X2) and (X3). Let $\Phi:=\{\phi_\alpha:X_\alpha\rightarrow \mathbb{R}\}_{\alpha\in J}$ be a family of measurable functions such that for all $\alpha\in (\inf J,\alpha^*]$ we have $\int|\phi_\alpha| d \phisicalalpha<\infty$, $\int |\phi_{\alpha^*}-\phi_{\alpha^*}(p_{\alpha^*})| d\nu_{\alpha^*}<\infty$ and 
\begin{align}\label{eq assumption cont int phi}
\lim_{\alpha\to\alpha^*-0}\int (\phi_{\alpha}-\phi_{\alpha^*}(p_{\alpha^*}))d\srbalpha=\int (\phi_{\alpha^*}-\phi_{\alpha^*}(p_{\alpha^*}))d\nu_{\alpha^*}.
\end{align}

Note that, since $\mu_{\alpha^*}=\delta_{p_{\alpha^*}}$, we have $\int \phi_{\alpha^*} d\mu_{\alpha^*}=\phi_{\alpha^*}(p_{\alpha^*})$.
By \eqref{eq kac}, for all $\alpha\in J$ with $\alpha<\alpha^*$ we have
\begin{align}\label{eq kac for phi}
    \int \phi_{\alpha} d\phisicalalpha-\int \phi_{\alpha^*} d\mu_{\alpha^*}=\int(\phi_{\alpha}-\phi_{\alpha^*} (p_{\alpha^*})) d\phisicalalpha=\frac{\int (\phi_{\alpha}-\phi_{\alpha^*}(p_{\alpha^*}))d\srbalpha}{\int \returnalpha d\srbalpha}.
\end{align}
By \eqref{eq cal return} and \eqref{eq lower and upper bound for sum of tail}, for all $\alpha\in J$ with $\alpha<\alpha^*$ we obtain
\begin{align}\label{eq estimate for return}
    D^{-1}\zeta(v(\alpha)+u(\alpha))+1
    \leq 
    \int \returnalpha d\srbalpha
    \leq
    D\zeta(v(\alpha)-u(\alpha))+1.
\end{align}
Since $z \in \mathbb{C}\mapsto \zeta(z)$ has a simple pole at $1$ with residue $1$ and 
\begin{align}\label{eq limit v and u}
\lim_{\alpha\to\alpha^*-0}(v(\alpha)-u(\alpha))=\lim_{\alpha\to\alpha^*-0}(v(\alpha)+u(\alpha))=1    
\end{align}
by \eqref{eq X3}, we obtain 
\begin{align*}
    &\lim_{\alpha\to\alpha^*-0}(v(\alpha)-u(\alpha)-1)\zeta(v(\alpha)-u(\alpha))
    \\&=    \lim_{\alpha\to\alpha^*-0}(v(\alpha)+u(\alpha)-1)\zeta(v(\alpha)+u(\alpha))=1.
\end{align*}
Combining this with \eqref{eq X3}, \eqref{eq kac for phi}, \eqref{eq estimate for return}, \eqref{eq assumption cont int phi} and \eqref{eq limit v and u}, we obtain  
\begin{align*}
&D^{-1}v_-'(\alpha^*){\int (\phi_{\alpha^*}-\phi_{\alpha^*}(p_{\alpha^*}))d\nu_{\alpha^*}}
\leq \liminf_{\alpha\to\alpha^*-0}
\frac{\int \phi_\alpha d\phisicalalpha-\int \phi_{\alpha^*} d\mu_{\alpha^*}}{\alpha-\alpha^*}
\\&\leq\limsup_{\alpha\to\alpha^*-0}\frac{\int \phi_\alpha d\phisicalalpha-\int \phi_{\alpha^*} d\mu_{\alpha^*}}{\alpha-\alpha^*}
\leq 
D
v_-'(\alpha^*){\int (\phi_{\alpha^*}-\phi_{\alpha^*}(p_{\alpha^*}))d\nu_{\alpha^*}}. 
\end{align*}
We now prove \eqref{eq abstract main eq} assuming (X3'). Combining \eqref{eq cal return} with \eqref{eq expansion of tail abstruct}, for all $\alpha\in J$ with $\alpha<\alpha^*$ we obtain 
\begin{align*}
    \int \returnalpha d\srbalpha=\sum_{k=0}^\infty \srbalpha(\{\returnalpha>k\})
    =c_\alpha\zeta(v(\alpha))+ O\left(\zeta(v(\alpha)+\epsilon)\right)
\end{align*}
Therefore, 
since $z \in \mathbb{C}\mapsto \zeta(z)$ is continuous on $\{z\in\mathbb{C}:\text{Re}(z)>1\}$ and has a simple pole at $1$ with residue $1$, $\lim_{\alpha\to \alpha^*} c_\alpha=c_{\alpha^*}$ and 
$
\lim_{\alpha\to\alpha^*}v(\alpha)=1    
$
by \eqref{eq X3'}, we obtain 
\begin{align*}
    \lim_{\alpha\to\alpha^*-0}(v(\alpha)-1)\int \returnalpha d\srbalpha
    =c_{\alpha^*}.
\end{align*}
Combining this limit with \eqref{eq kac for phi} and \eqref{eq X3'}, we obtain 
\begin{align*}\lim_{\alpha\to\alpha^*-0}\frac{\int \phi_\alpha d\phisicalalpha-\int \phi_{\alpha^*} d\mu_{\alpha^*}}{\alpha-\alpha^*}
=
\frac{v_-'(\alpha^*)\int (\phi_{\alpha^*}-\phi_{\alpha^*}(p_{\alpha^*}))d\nu_{\alpha^*}}{c_{\alpha^*}}. 
\end{align*}
\qed

\emph{Proof of Corollary \ref{cor time average}.}
By Birkhoff's ergodic theorem, for each $\alpha\in J$ with  $\alpha<\alpha^*$, we have 
\begin{align}\label{eq birkkhoff average}
\lim_{n\to\infty}\frac{1}{n}\sum_{i=0}^{n-1}\phi_\alpha (\mapalpha^i(x))=\int \phi_\alpha d \phisicalalpha.    
\end{align}
for $\phisicalalpha$-almost every  $x\in X$.
Since $\phisicalalpha$ is equivalent to $m_\alpha$, for each $\alpha\in J$ with  $\alpha<\alpha^*$, equality \eqref{eq birkkhoff average} holds for $m_\alpha$-almost every  $x\in X_\alpha$. This implies that for each $\alpha\in J$ with  $\alpha<\alpha^*$ we obtain 
\begin{align*}
&\operatorname*{ess\,lim}_{n\to\infty}
\left|
\frac{\frac{1}{n}\sum_{i=0}^{n-1} \phi_\alpha(\mapalpha^i(x))-\phi_{\alpha^*}(p_{\alpha^*})}{\alpha-\alpha^*}-v_-'({\alpha^*})c_{\alpha^*}^{-1}\int (\phi_{\alpha^*}-\phi_{\alpha^*}(p_{\alpha^*}))d\nu_{\alpha^*}\right|
\\&=
\left|
\frac{\int \phi_\alpha d\phisicalalpha-\phi_{\alpha^*}(p_{\alpha^*})}{\alpha-\alpha^*}-v_-'({\alpha^*})c_{\alpha^*}^{-1}\int (\phi_{\alpha^*}-\phi_{\alpha^*}(p_{\alpha^*}))d\nu_{\alpha^*}\right|.
\end{align*}
Corollary \ref{cor time average} now follows from Theorem \ref{thm main super abstract}.   \qed

\section{Examples of families of one-dimensional intermittent maps}\label{sec example}
Let $I:=[0,1]$ be endowed with the Euclidean topology.
For any subset $J\subset \mathbb{R}$ we always endow $J$ with the relative topology from the Euclidean space $\mathbb{R}$ and the Borel $\sigma$-algebra $\mathcal{B}$.
For a non-trivial interval $J\subset \mathbb{R}$ and $i\in\mathbb{N}$, 
the class $C^i(J)$ of $C^i$ functions on $J$ is endowed with the $C^i$ norm defined by
\[
\|h\|_{C^i(J)} := \max\{\|h\|_\infty,\|h'\|_\infty,\ldots,\|h^{(i)}\|_\infty\}
\quad \text{for } h\in C^i(J).
\]

Let $0<\alpha_-<\alpha_+$. For $A\subset\mathbb{R}$ we denote by $\overline{A}$ the Euclidean closer of $A$. 
In this paper, we consider a family of intermittent maps $\{f_\alpha\}_{\alpha\in[\alpha_-,\alpha_+]}$ on $I$ satisfying the following conditions:
\begin{itemize}
    \item[(f1)] There exist $m\in\mathbb{N}$ with $m\geq 2$ and disjoint non-trivial intervals $\{I_i\}_{1\leq i\leq m}$ such that $I\setminus \bigcup_{i=1}^m I_i$ is a null set with respect to the Lebesgue measure on $I$.
    \item[(f2)] For all $\alpha\in [\alpha_-,\alpha_+]$ and $1\leq i\leq m$ the map $f_{\alpha}|_{I_i}:I_i\rightarrow f(I_i)$ is a $C^3$ diffeomorphism and $\overline{f(I_i)}=[0,1]$. Furthermore, 
    there exists a open set $W_{\alpha,i}$ such that $\overline{I_i}\subset W_{\alpha,i}$ and $f_\alpha|_{I_i}$ extends to a $C^3$ diffeomorphism $f_{\alpha,i}$ from $W_{\alpha,i}$ onto its images.
    Moreover, for each $1\leq i\leq m$ the functions 
    \begin{align}\label{eq jointly}
    (\alpha,x)\in[\alpha_-,\alpha_+]\times \overline{I_i}\mapsto f_{\alpha,i}'(x):= \frac{d}{dx}f_{\alpha,i}(x)    
    \end{align}
    and 
    \begin{align}\label{eq jointly twice}
    (\alpha,x)\in[\alpha_-,\alpha_+]\times \overline{I_i}\mapsto f_{\alpha,i}''(x):= \frac{d^2}{dx^2}f_{\alpha,i}(x)    
    \end{align}
    are jointly continuous.
    \item[(f3)] There exists $\sigma>1$ such that for all $\alpha\in [\alpha_-,\alpha_+]$ and $2\leq i\leq m$ we have $\inf_{x\in I_i} |\frac{d}{dx}f_{\alpha,i}(x)|>\sigma$. Moreover, $f_{\alpha,1}(0)=0$ and for each $x\in \overline{I_1}\setminus\{0\}$ we have $\frac{d}{dx}f_{\alpha,1}(x)>1$.
    \item[(f4)] There exist a continuous function $\alpha\mapsto b_\alpha\in(0,\infty)$ on $[\alpha_-,\alpha_+]$, $\epsilon>0$ and a constant $C>0$ such that for all $\alpha\in[\alpha_-,\alpha_+]$ and $x\in I_1$ we have 
    \begin{align}\label{eq expansion}
        \left|\frac{d}{dx}f_{\alpha,1}(x)-(1+b_\alpha(1+\alpha)x^{\alpha})\right|\leq Cx^{\alpha+\epsilon}.
    \end{align}
    \item[(f5)] For all $1\leq i\leq m$ and $x\in \overline{I_i}$ the functions $\alpha\mapsto f_{\alpha,i}^{-1}(x)$, $\alpha\mapsto (f_{\alpha,i}^{-1})'(x):=\frac{d}{dx}f_{\alpha,i}^{-1}(x)$ and $\alpha\mapsto (f_{\alpha,i}^{-1})''(x):=\frac{d^2}{dx^2}f_{\alpha,i}^{-1}(x)$ are in $C^1([\alpha_-,\alpha_+])$.
\end{itemize}
Note that,
without loss of generality, 
we may assume that $\epsilon>0$ appearing in (f4) satisfies $\epsilon<\alpha_-$.

Let $\{f_\alpha\}_{\alpha\in[\alpha_-,\alpha_+]}$ be a family of maps on $I$ satisfying the above conditions. Define
\[
\indudomain:=\bigcup_{i=2}^m\overline{I}_i.
\]
Let $\alpha\in[\alpha_-,\alpha_+]$. We define the first return time function $\returnalpha:\indudomain\rightarrow \mathbb{N}\cup\{\infty\}$ of $f_{\alpha}$ by 
$
\returnalpha(x):=\inf\{n\in\mathbb{N}:f_\alpha^n(x)\in\indudomain\}.
$
and the first return map $\indumapalpha:\{\returnalpha<\infty\}\rightarrow \indudomain$ by 
\[
\indumapalpha(x):=f_\alpha^{\returnalpha(x)}(x).
\]
Let $2\leq i \leq m$ and let $k\in\mathbb{N}$. We set
\[
C_{(i,k)}:=\overline{I_i\cap \{\tau_\alpha=k\}}.
\]
We define the map $F_{\alpha,(i,k)}^{-1}:\indudomain\rightarrow C_{(i,k)}$ by 
\[
F_{\alpha,(i,k)}^{-1}(x):=F_{(i,k)}^{-1}(\alpha,x):=
(f_{\alpha,i}^{-1}\circ f_{\alpha,1}^{-(k-1)})(x).
\]
Following \cite{Korepanov}, we introduce the following notations:
We define $G_{\alpha,(i,k)}:\indudomain\rightarrow (0,\sigma^{-1}]$, $G_{\alpha,(i,k)}':\indudomain\rightarrow (0,\infty)$ and $G_{\alpha,(i,k)}'':\indudomain\rightarrow (0,\infty)$ by 
\begin{align*}
    &G_{\alpha,(i,k)}(x):=G_{(i,k)}(\alpha,x):=\left|\frac{d}{dx}F_{\alpha,(i,k)}^{-1}(x)\right|,\ 
    \\&G_{\alpha,(i,k)}'(x):=G_{(i,k)}'(\alpha,x):=\left|\frac{d^2}{dx^2}F_{\alpha,(i,k)}^{-1}(x)\right|
    \text{ and }
    \\&G_{\alpha,(i,k)}''(x):=G_{(i,k)}''(\alpha,x):=\left|\frac{d^3}{dx^3}F_{\alpha,(i,k)}^{-1}(x)\right|.
\end{align*}
For $x\in \indudomain$ we define 
\begin{align*}
&\partial_\alpha F_{\alpha,(i,k)}^{-1}(x):=\frac{\partial}{\partial \alpha} F^{-1}_{(i,k)}(\alpha,x),\ \partial_\alpha G_{\alpha,(i,k)}(x):=\frac{\partial}{\partial \alpha} G_{(i,k)}(\alpha,x), \text{ and }
\\&
\partial_\alpha G'_{\alpha,(i,k)}(x):=\frac{\partial}{\partial \alpha} G_{(i,k)}'(\alpha,x).
\end{align*}

The following conditions were considered in \cite{Korepanov}, and we assume that the family of induced systems $\{F_{\alpha}\}_{\alpha\in[\alpha_-,\alpha_+]}$ derived from $\{f_\alpha\}_{\alpha\in[\alpha_-,\alpha_+]}$ satisfies these conditions:
There exist constants $K_0>0$ and $\{\gamma_{(i,k)}\}_{2\leq i\leq m, k\in\mathbb{N}}\subset[1,\infty)$ such that uniformly in $\alpha\in[\alpha_-,\alpha_+]$ and $(i,k)\in \{2,\cdots,m\}\times \mathbb{N}$ we have 
\begin{itemize}
    \item[(A2)] $\|G_{\alpha,(i,k)}'/G_{\alpha,(i,k)}\|_{\infty}\leq K_0$.
    \item[(A3)] $\|G_{\alpha,(i,k)}''/G_{\alpha,(i,k)}\|_{\infty}\leq K_0$.
    \item[(A4)] $\|\partial_\alpha F^{-1}_{\alpha,(i,k)}\|_\infty\leq \gamma_{(i,k)}$.
    \item[(A5)] $\|\partial_\alpha G_{\alpha,(i,k)}/G_{\alpha,(i,k)}\|_\infty\leq \gamma_{(i,k)}$.
    \item[(A6)] $\|\partial_\alpha G_{\alpha,(i,k)}'/G_{\alpha,(i,k)}\|_\infty\leq \gamma_{(i,k)}$. 
    \item[(A7)] $\sum_{i=2}^m\sum_{k=1}^\infty \|G_{\alpha,(i,k)}\|_\infty\gamma_{(i,k)}\leq K_0$.
\end{itemize}
Note that by (f3), for all $\alpha\in[\alpha_-,\alpha_+]$ and $(i,k)\in \{2,\cdots,m\}\times \mathbb{N}$ we have  
\begin{itemize}
\item[(A1)] $\|G_{\alpha,(i,k)}\|_\infty\leq \sigma^{-1}$.   
\end{itemize}
We denote by $\maps(I,\alpha_-,\alpha_+)$ the set of all families $\{f_\alpha:I\rightarrow I\}_{\alpha\in[\alpha_-,\alpha_+]}$ satisfying (f1)-(f5) and such that $\{F_\alpha\}_{\alpha\in J}$ satisfies (A2)-(A7).

Let $\{f_\alpha\}_{\alpha\in[\alpha_-,\alpha_+]} \in \maps(I,\alpha_-,\alpha_+)$. For $\alpha\in [\alpha_-,\alpha_+]$, a function $u:\indudomain\rightarrow\mathbb{R}$ we define the Ruelle operator $P_\alpha$ by 
\[
P_\alpha(u)(x):=\sum_{i=2}^m\sum_{k=1}^\infty G_{\alpha,(i,k)}(x) u( F^{-1}_{\alpha,(i,k)}(x))\ (x\in \indudomain),
\]
whenever the infinite sum on the right-hand side converges.
Note that by (A7), for all continuous function $u$ on $\indudomain$ we have 
\[
\sup_{x\in\indudomain}\left\{\sum_{i=2}^m\sum_{k=1}^\infty |G_{\alpha,(i,k)}(x) u( F^{-1}_{\alpha,(i,k)}(x))|\right\}
\leq \|u\|_\infty\sum_{i=2}^m\sum_{k=1}^\infty \|G_{\alpha,(i,k)}\|_\infty\leq K_0
\]
and thus, $P_\alpha(u)(x)$ is well-defined for all $x\in \indudomain$.
By \cite[Lemma 4.1]{Korepanov}, the operator $P_\alpha: C^2(\indudomain)\rightarrow C^2(\indudomain)$ is well-defined.  
Moreover, \cite{Korepanov} showed the following theorem:
Define  
\[
\tilde\lambda:=(\lambda(\indudomain))^{-1}\lambda|_{\indudomain}.
\]
\begin{thm}\label{thm korepanov}
    Let $0<\alpha_-<\alpha_+$  and let $\{f_\alpha\}_{\alpha\in[\alpha_-,\alpha_+]} \in \maps(I,\alpha_-,\alpha_+)$. Then, for all $\alpha\in [\alpha_-,\alpha_+]$ there exists a unique function $\indudensityalpha\in C^2(\indudomain)$ such that $\indusrbalpha(\indudomain)=1$, where $\indusrbalpha:=\indudensityalpha d\indulebesgue$,  
    \begin{align}\label{eq fixed point transfer operator}
    P_\alpha(\indudensityalpha)=\indudensityalpha
    \end{align}
    and there exists a constant $D_\alpha\geq1$ such that for all $x\in \indudomain$ we have 
    \begin{align}\label{eq nonnegative}
        D_\alpha^{-1}\leq \indudensityalpha(x)\leq D_\alpha.
    \end{align}
    Moreover, 
    the map $\alpha\in [\alpha_-,\alpha_+]\mapsto \indudensityalpha\in C^2(\indudomain)$ 
    is continuous.
\end{thm}
In the following, for each $\alpha\in [\alpha_-,\alpha_+]$ we denote by $\indudensityalpha$ the function obtained in Theorem \ref{thm korepanov}.
By \eqref{eq fixed point transfer operator}, \eqref{eq nonnegative} and the definition of the transfer operator $P_\alpha: C^2(\indudomain)\rightarrow C^2(\indudomain)$, it is straightforward to show that for each $\alpha\in [\alpha_{-},\alpha_{+}]$, the Bore probability measure 
\[
\indusrbalpha:=\indudensityalpha d\tilde\lambda,
\]
is $\indumapalpha$-invariant, ergodic and equivalent to $\tilde\lambda$ (see, for example, the arguments in \cite[Sections 17 and 18]{Urbanskinoninvertible}). 

Let $\alpha\in [\alpha_-,\alpha_+]$.
We define the measure $\srbalpha$ on $I$ by 
\begin{align}\label{eq def srb one dimensional}
\srbalpha:=\sum_{n=0}^{\infty}\sum_{k=n+1}^\infty\indusrbalpha|_{\{\returnalpha=k\}}\circ \mapalpha^{-n}.    
\end{align}
By \cite{Thaler1980} (see also Proposition \ref{prop expansion of tail}), $\srbalpha(I)=\infty$ if and only if $\alpha\geq 1$. We define the probability measure $\phisicalalpha$ on $I$ by 
\begin{align}\label{eq def phisical abstruct} 
\phisicalalpha:=
 \left\{
 \begin{array}{cc}
({\srbalpha(I)})^{-1} {\srbalpha}  & \text{if}\  \alpha_-\leq\alpha<1\\
   \delta_0   & \text{if}\ 1\leq\alpha\leq \alpha_+
 \end{array}
 \right. .
\end{align}

\begin{rem}\label{rem verify LSv map}
Let $0<\alpha_-<\alpha_+$. 
    We consider the family of maps $\{\lsvmapalpha\}_{\alpha\in [\alpha_-,\alpha_+]}$, where $\lsvmapalpha$ is defined by \eqref{eq def lsv map intro} for each $\alpha\in [\alpha_-,\alpha_+]$. It is not difficult to verify that $\{\lsvmapalpha\}_{\alpha\in [\alpha_-,\alpha_+]}$ satisfies (f1)-(f5) with $m=2$, $I_1=[0,1/2]$, $I_2=(1/2,1]$ and $b_\alpha=2^\alpha$ for $\alpha\in [\alpha_-,\alpha_+]$. Moreover, by essentially the same calculation in the proof of \cite[Theorem 3.1]{Korepanov} (see also \cite[Section 5]{BahsounSaussol}) based on a weak version of the estimate in Lemma \ref{lemma estimate yn}, $\{\lsvmapalpha\}_{\alpha\in [\alpha_-,\alpha_+]}$ satisfies (A2)-(A7). Therefore, $\{\lsvmapalpha\}_{\alpha\in [\alpha_-,\alpha_+]}\in \maps(I,\alpha_-,\alpha_+)$. 
\end{rem}

For $\eta>0$ a function $\phi: I\rightarrow \mathbb{R}$ is said to be H\"older continuous with exponent $\eta$ if there exists a constant $C>0$ such that for all $x,y\in I$ we have 
\begin{align}\label{eq def holder}
    |\phi(x)-\phi(y)|\leq C|x-y|^\eta.
\end{align}
A function $\phi:I\rightarrow \mathbb{R}$ is said to be H\"older continuous if there exists $\eta>0$ such that $\phi$ is H\"older continuous with exponent $\eta$.

We are now in the position to state our main theorem in this section.
\begin{thm}\label{thm main abstract}
 Let $\alpha_-<1$ and let $\alpha_+>1$. Let $\{\mapalpha\}_{\alpha\in[\alpha_-,\alpha_+]}\in \maps(I,\alpha_-,\alpha_+)$. Then, $\{\mapalpha\}_{\alpha\in[\alpha_-,\alpha_+]}$ satisfies (X1), (X2) and (X3') with 
 \[
 c_\alpha=\frac{\sum_{i=2}^m\indudensityalpha(f_{\alpha,i}^{-1}(0))|(f_{\alpha,i}^{-1})'(0)|}{(\alpha b_\alpha)^{1/\alpha}}
 \text{ and }
 v(\alpha)=\frac{1}{\alpha}.
 \]
Moreover, for any H\"older continuous function $\phi: I\rightarrow \mathbb{R}$ we have
\[
    \int|\phi-\phi (0)| d\srb<\infty \text{ and }\lim_{\alpha\to 1-0}\int (\phi-\phi(0)) d\srbalpha=\int(\phi-\phi(0)) d\srb.
\]
In particular, the conclusions of Theorem \ref{thm main super abstract} and Corollary \ref{cor time average} hold.
\end{thm}

\section{Proof of Theorem \ref{thm main abstract}}\label{sec proof sec 3}
In this section, we give the proof of Theorem \ref{thm main abstract}. 
Let $0<\alpha_-<\alpha_+$. For a family of maps $\{\mapalpha\}_{\alpha\in[\alpha_-,\alpha_+]}$ satisfying (f1)-(f4) and $\alpha\in [\alpha_-,\alpha_+]$ we define $y_{\alpha,0}:=1$ and, for all $n\in\mathbb{N}$, 
\begin{align}\label{eq def y}
    y_{\alpha,n}:=\map_{\alpha,1}^{-1}(y_{\alpha,n-1}).
\end{align}
\begin{lemma}\label{lemma estimate yn}
    Let $0<\alpha_-<\alpha_+$. Let $\{\mapalpha\}_{\alpha\in[\alpha_-,\alpha_+]}$ be  a family of maps satisfying (f1)-(f4). Then, we have
    \begin{align*}
        \yan-\left(\frac{1}{\alpha b_\alpha n}\right)^{1/\alpha}
        = 
        O\left(\frac{1}{n^{(1+\epsilon)/\alpha}}\right)
    \end{align*}
\end{lemma}
\begin{proof}
    We notice that by the continuity of the function $\alpha \in[\alpha_-,\alpha_+]\mapsto b_\alpha$ (see (f4)), for all $\alpha\in[\alpha_-,\alpha_+]$ we have 
    \begin{align}\label{eq bounded b}
       0<\underline{b}:=\min\{b_\alpha:\alpha\in [\alpha_-,\alpha_+]\} \leq b_\alpha \leq\max\{b_\alpha:\alpha\in [\alpha_-,\alpha_+]\}=: \overline{b}<\infty.
    \end{align}
    Let $\alpha\in[\alpha_-,\alpha_+]$ and let $n\in\mathbb{N}$.
    By \eqref{eq expansion}, for all $x\in[0,1]$ we have 
    $
    f_{\alpha,1}(x)=x+b_\alpha x^{1+\alpha}+O(x^{1+\alpha+\epsilon}).
    $ Therefore, by the definition of $\yan$,  we obtain
    \[
    {\yan}^{-1}={\yanp}^{-1}(1+b_\alpha \yanp^{\alpha}+O(\yanp^{\alpha+\epsilon}))^{-1}.
    \]
    Thus, by applying Taylor's theorem to the function $x\mapsto(1+x)^{-1}$ at $x=0$ and using \eqref{eq bounded b},  we obtain
    ${\yan}^{-1}={\yanp}^{-1}(1-b_\alpha \yanp^{\alpha}+O(\yanp^{\alpha+\epsilon})) 
    $
    and hence, 
    \[
    \uan=\uanp\left(1-{b_\alpha}{\uanp^{-1}}+O\left({\uanp^{-(1+\epsilon/\alpha)}}\right)\right)^\alpha,
    \text{ where }\uan:=\yan^{-\alpha}.
    \]
    Moreover, by applying Taylor's theorem to the function $x\mapsto x^{\alpha}$ at $x=1$ and using \eqref{eq bounded b} and the boundedness of the set $[\alpha_-,\alpha_+]$, we obtain
    \[
    \uan=\uanp\left(1-\alpha b_\alpha \uanp^{-1}+O\left({\uanp^{-(1+\epsilon/\alpha)}}\right)\right)
    \]
   By applying a telescoping argument, for all $n\in\mathbb{N}$ we obtain 
\begin{align}\label{eq n and u expansion}
    \uan=\sum_{k=0}^{n-1}
(u_{\alpha,k+1}-u_{\alpha,k})+u_{\alpha,0}=n\alpha b_\alpha+O\left(\sum_{k=0}^{n-1}{u_{\alpha,k+1}^{-\epsilon/\alpha}}\right)+1.
\end{align}
We shall show that there exists $N\in\mathbb{N}$ and $C_0>1$ such that for all $\alpha\in [\alpha_-,\alpha_+]$ and $n\geq N$ we have 
\begin{align}\label{eq n and u}
   \frac{1}{C_0} \leq \frac{\uan}{n}\leq C_0.
\end{align}
By \eqref{eq bounded b}, for all $\alpha\in [\alpha_-,\alpha_+]$ we have
\begin{align}
    \underline{b} \alpha_-\leq b_\alpha\alpha\leq \overline{b}\alpha_+.
\end{align}
By the definition of $O$, there exists $D>1$ such that for all $n\in\mathbb{N}$ and $\alpha\in[\alpha_-,\alpha_+]$ we have
\begin{align}\label{eq big O sum}
    \left|O\left(\sum_{k=0}^{n-1}{u_{\alpha,k+1}^{-\epsilon/\alpha}}\right)\right|\leq D\sum_{k=0}^{n-1}{u_{\alpha,k+1}^{-\epsilon/\alpha}}=D\sum_{k=1}^{n} y_{\alpha,k}^{\epsilon}. 
\end{align}
We take a small positive number $\xi$ satisfying $0<\xi<\min\{1,\underline{b}\alpha_-/(4D)\}$.
By the mean value theorem, for all $\alpha\in[\alpha_-,\alpha_+]$ there exists $x_\alpha\in (0,\xi)$ such that 
\begin{align*}
    f_{\alpha,1}^{-1}(\xi)=f_{\alpha,1}^{-1}(\xi)-f_{\alpha,1}^{-1}(0)=(f_{\alpha,1}^{-1})'(x_\alpha)\xi=\left(f_{\alpha,1}'(f_{\alpha,1}^{-1}(x_\alpha))\right)^{-1}\xi.
\end{align*}
Therefore, by the joint continuity of the function defined by \eqref{eq jointly}, there exists $\eta>0$ such that for all $\alpha\in [\alpha_-,\alpha_+]$ we have 
\begin{align}\label{eq proof def eta}
    \eta<f_{\alpha,1}^{-1}(\xi).
\end{align}
Again, by the joint continuity of the function defined by \eqref{eq jointly}, Berge's Maximum Theorem yields that the function 
\[
\alpha\in[\alpha_-,\alpha_+]\mapsto \varphi(\alpha):= \min\left\{ f_{\alpha,1}'(x):x\in \overline{I_1}\setminus[0,\eta)\right\}.
\]
is continuous. Since for all $x\in \overline{I_1}\setminus\{0\}$ we have $ f_{\alpha,1}'(x)>1$, this yields that 
\begin{align}\label{eq partially expanding}
1<s:=\min\{\varphi(\alpha):\alpha\in [\alpha_-,\alpha_+]\}.    
\end{align}
We will show that 
\begin{align}\label{eq claim}
&\text{
    there exist $ L \geq 1$ such that for all $k\geq  L $ and $\alpha\in[\alpha_-,\alpha_+]$}
    \\&\text{ we have $y_{\alpha,k}\in [0,2\xi)$.} \nonumber
\end{align}
For a contradiction, we assume that the claim does not hold. 
Then, there exist a sequence $\{n_{j}\}_{j\in\mathbb{N}}$ and $\{\alpha_{j}\}_{j\in\mathbb{N}}\subset [\alpha_-,\alpha_+]$ such that for all $j\in\mathbb{N}$ we have $n_j<n_{j+1}$ and 
\begin{align}\label{eq hyp contra}
y_{\alpha_j,n_j}\geq2\xi. 
\end{align}
Notice that, by (f3), for all $\alpha\in[\alpha_-,\alpha_+]$ and $x\in[0,1]$ we have 
\begin{align}\label{eq ine alpha,1 inverse}
    f_{\alpha,1}^{-1}(x)<x.
\end{align} 
For all $j\in\mathbb{N}$ and $1\leq k\leq n_j$ we have 
\begin{align}\label{eq exponential convergence}
   y_{\alpha_j,k}-f^{-1}_{\alpha_j,1}(\xi)\leq s^{-k}.
\end{align}
The proof of this claim proceeds as follows: Let $j\in\mathbb{N}$. By \eqref{eq def y} and the mean value theorem, there exists $x_{j,1}\in (\xi, 1)$ such that
\[
y_{\alpha_j,1}-f^{-1}_{\alpha_j,1}(\xi)=f^{-1}_{\alpha_j,1}(1)-f^{-1}_{\alpha_j,1}(\xi)=(f^{-1}_{\alpha_j,1})'(x_{j,1})(1-\xi).
\]
Since $f^{-1}_{\alpha_j,1}(x_{j,1})\in (f^{-1}_{\alpha_j,1}(\xi), y_{\alpha_j,1})\subset \overline{I_1}\setminus[0,\eta)$ by \eqref{eq proof def eta}, we obtain 
\[
y_{\alpha_j,1}-f^{-1}_{\alpha_j,1}(\xi)=(f^{-1}_{\alpha_j,1})'(x_{j,1})(1-\xi)
=\left({f_{\alpha_j,1}'(f^{-1}_{\alpha_j,1}(x_{j,1}))}\right)^{-1}(1-\xi)\leq s^{-1}.
\]
Let $1\leq k\leq n_j-1$
Suppose that \eqref{eq exponential convergence} holds for $k$.
By \eqref{eq def y} and the mean value theorem, there exists $x_{j,k+1}\in (\xi, y_{\alpha_j,k})$ such that
\[
y_{\alpha_j,k+1}-f^{-1}_{\alpha_j,1}(\xi)=f^{-1}_{\alpha_j,1}(y_{\alpha_j,k})-f^{-1}_{\alpha_j,1}(\xi)=(f^{-1}_{\alpha_j,1})'(x_{j,{k+1}})(y_{\alpha_j,k}-\xi).
\]
Since $f^{-1}_{\alpha_j,1}(x_{j,k+1})\in (f^{-1}_{\alpha_j,1}(\xi), y_{\alpha_j,k+1})\subset \overline{I_1}\setminus[0,\eta)$ by \eqref{eq proof def eta}, \eqref{eq ine alpha,1 inverse} implies that 
\[
y_{\alpha_j,k+1}-f^{-1}_{\alpha_j,1}(\xi)=
\left(f_{\alpha_j,1}'(f_{\alpha_j,1}^{-1}(x_{j,{k+1}}))\right)^{-1}(y_{\alpha_j,k}-\xi)\leq s^{-1}(y_{\alpha_j,k}-f^{-1}_{\alpha_j,1}(\xi)).
\]
Therefore, by the induction hypothesis, we obtain $y_{\alpha_j,k+1}-f^{-1}_{\alpha_j,1}(\xi)\leq s^{-(k+1)}$. Hence, by the induction, we obtain \eqref{eq exponential convergence} for all $1\leq k\leq n_j$. 

Since $n_j<n_{j+1}$ for all $j\in\mathbb{N}$, there exists $J\in\mathbb{N}$ such that $s^{-n_J}<\xi$. Therefore, by \eqref{eq hyp contra} and \eqref{eq ine alpha,1 inverse}, we obtain
\[
\xi=2\xi-\xi<y_{\alpha_J,n_J}-f^{-1}_{\alpha_J,1}(\xi)\leq s^{-n_J}<\xi.
\]
This is a contradiction. Thus, we obtain \eqref{eq claim}.

Let $N_0$ be a positive integer satisfying
${L}\leq \xi N_0.$
Then, by \eqref{eq big O sum} and \eqref{eq claim}, for all $n\geq N_0$ and $\alpha\in[\alpha_-,\alpha_+]$ we obtain
\begin{align*}
\frac{1}{n}
    \left|O\left(\sum_{k=0}^{n-1}{u_{\alpha,k+1}^{-\epsilon/\alpha}}\right)\right|
    \leq D\left(
    \sum_{k=1}^{L}\frac{y_{\alpha,k}^{\epsilon}}{n}
    +\sum_{k=L+1}^{n}\frac{y_{\alpha,k}^{\epsilon}}{n}\right)
    \leq D\left(\frac{L}{n}+2\xi\right) 
    \leq 3D\xi.
\end{align*}
By \eqref{eq n and u expansion}, this implies that for all $n\geq N_0$ and $\alpha\in[\alpha_-,\alpha_+]$ we have
\[
0<\alpha_- \underline{b}-4D\xi\leq \frac{\uan}{n}\leq \alpha_+ \overline{b} +4D\xi.
\]
and thus, we obtain \eqref{eq n and u}. Therefore, since $\epsilon<\alpha_-$ we obtain 
$O(\sum_{k=0}^{n-1}{u_{\alpha,k+1}^{-\epsilon/\alpha}})=O(\sum_{k=1}^{n}k^{-\epsilon/\alpha})=O(n^{1-\epsilon/\alpha})$. 
Combining this with \eqref{eq n and u expansion} and applying Taylor's theorem to the function $x\mapsto x^{1/\alpha}$ at $x=1$, for all $n\in\mathbb{N}$ and $\alpha\in[\alpha_-,\alpha_+]$ we obtain
\[
\yan=\left(\frac{1}{\alpha b_\alpha n}\right)^{1/\alpha}\left(\frac{1}{1+O(n^{-\epsilon/\alpha})}\right)^{1/\alpha}=\left(\frac{1}{\alpha b_\alpha n}\right)^{1/\alpha}\left({1+O(n^{-\epsilon/\alpha})}\right).
\]
\end{proof}

\begin{prop}\label{prop expansion of tail}
 Let $0<\alpha_-<\alpha_+$. Let $\{\mapalpha\}_{\alpha\in[\alpha_-,\alpha_+]}\in \maps(I,\alpha_-,\alpha_+)$. Then, for all $\alpha\in [\alpha_-,\alpha_+]$ and $n\in\mathbb{N}$ 
 we have 
 \[
 \srbalpha(\{\returnalpha>n\})=\frac{\sum_{i=2}^m\indudensityalpha(f_{\alpha,i}^{-1}(0))|(f_{\alpha,i}^{-1})'(0)|}{(\alpha b_\alpha)^{1/\alpha}}n^{-1/\alpha}+O(n^{-(1+\epsilon)/\alpha}).
 \]
 \begin{proof}
     Let $\alpha\in [\alpha_-,\alpha_+]$ and let $n\in\mathbb{N}$. Since $\overline{\{\returnalpha>n\}}=\cup_{i=2}^mf_{\alpha,i}^{-1}([0,\yan])$, we have 
     \begin{align}\label{eq pr prop zero}
        & \srbalpha(\{\returnalpha>n\})=
         \sum_{i=2}^m X_{\alpha,n,i} 
         +\sum_{i=2}^mY_{\alpha,n,i}, \text{ where } 
     \\&X_{\alpha,n,i}:=
         \int_{f_{\alpha,i}^{-1}([0,\yan])}\left(\indudensityalpha(x)-\indudensityalpha\left(\map_{\alpha,i}^{-1}(0)\right)\right) dx
         \text{ and }\nonumber
        \\& Y_{\alpha,n,i}:=\indudensityalpha\left(\map_{\alpha,i}^{-1}(0)\right)|f_{\alpha,i}^{-1}(\yan)-f_{\alpha,i}^{-1}(0)|.\nonumber
     \end{align}
     By the mean value theorem, for each $2\leq i\leq m$ there exists $z_{\alpha,n,i}\in [0,\yan]$ such that $|f_{\alpha,i}^{-1}(\yan)-f_{\alpha,i}^{-1}(0)|=|(f_{\alpha,i}^{-1})'(z_{\alpha,n,i})|\yan$ 
     Since for each $2\leq i\leq m$ the sign $\text{sgn}((f_{\alpha,i}^{-1})')$ of $(f_{\alpha,i}^{-1})'$ is constant, for each $2\leq i\leq m$ we obtain
     \begin{align*}
         Y_{\alpha,n,i}
         &=\indudensityalpha\left(\map_{\alpha,i}^{-1}(0)\right)|(f_{\alpha,i}^{-1})'(z_{\alpha,n,i})|\yan
     \\&=\indudensityalpha\left(\map_{\alpha,i}^{-1}(0)\right)
     \text{sgn}((f_{\alpha,i}^{-1})')\left((f_{\alpha,i}^{-1})'(z_{\alpha,n,i})-(f_{\alpha,i}^{-1})'(0)\right)
     \yan
     \\&+
     \indudensityalpha\left(\map_{\alpha,i}^{-1}(0)\right)|(f_{\alpha,i}^{-1})'(0)|\yan.\nonumber
     \end{align*}
     By Theorem \ref{thm korepanov}, $D_0:=\sup_{\alpha\in [\alpha_-,\alpha_+]}\|\indudensityalpha\|_{C^2(\indudomain)}<\infty$.
The joint continuity of the function defined by \eqref{eq jointly twice} implies that 
\[
D_1:=\sup_{2\leq i\leq m}\sup_{\alpha\in[\alpha_-,\alpha_+]}\sup_{x\in I}|(f_{\alpha,i}^{-1})''(x)|<\infty.
\]
Thus, by the mean value theorem, we obtain 
\begin{align*}
\sum_{i=2}^m\indudensityalpha\left(\map_{\alpha,i}^{-1}(0)\right)|(f_{\alpha,i}^{-1})'(z_{\alpha,n,i})-(f_{\alpha,i}^{-1})'(0)|\yan\leq D_0D_1\sum_{i=2}^m\yan^2
    \end{align*}
which yields that 
\begin{align}\label{eq Y}
\sum_{i=2}^mY_{\alpha,n,i}
    =\sum_{i=2}^m\indudensityalpha\left(\map_{\alpha,i}^{-1}(0)\right)|(f_{\alpha,i}^{-1})'(0)|\yan+O(\yan^2).
\end{align}
     Moreover, the joint continuity of the function defined by \eqref{eq jointly} implies that $D_2:=\sup_{2\leq i\leq m}\sup_{\alpha\in [\alpha_-,\alpha_+]}\sup_{x\in [0,1]}|(f^{-1}_{\alpha,i})'(x)|<\infty$
     Thus, by the mean value theorem, for all $2\leq i\leq m$ we obtain 
     \begin{align*}
     &
         |X_{\alpha,n.i}|
         \leq D_0\int_{f_{\alpha,i}^{-1}([0,\yan])}\left|x-\map_{\alpha,i}^{-1}(0)\right| dx
         =\frac{D_0}{2}|f_{\alpha,i}^{-1}(\yan)-f_{\alpha,i}^{-1}(0)|^2\leq \tilde D \yan^2,
         \nonumber
     \end{align*}
     where $\tilde D:=D_0D_2^2$.
     This implies that $\sum_{i=2}^mX_{\alpha,n,i}=O(\yan^{2})$. Combining this with \eqref{eq pr prop zero} and \eqref{eq Y}, we obtain 
     \begin{align*}
         \indusrbalpha(\{\returnalpha>n\})=\sum_{i=2}^m\indudensityalpha\left(\map_{\alpha,i}^{-1}(0)\right)|(f_{\alpha,i}^{-1})'(0)|\yan+O(\yan^2).
     \end{align*}
Since $\epsilon<\alpha_-<1$, Lemma \ref{lemma estimate yn} implies that $O(\yan^2)=O(n^{-2/\alpha})=O(n^{-(1+\epsilon)/\alpha})$. Thus, Lemma \ref{lemma estimate yn} completes the proof.
      \end{proof}
\end{prop}

\emph{Proof of Theorem \ref{thm main abstract}.}
 Let $\alpha_-<1$ and let $\alpha_+>1$. Let $\{\mapalpha\}_{\alpha\in[\alpha_-,\alpha_+]}\in \maps(I,\alpha_-,\alpha_+)$ and let $\phi:I\rightarrow  \mathbb{R}$ be H\"older continuous.
 
 We first show that $\{\mapalpha\}_{\alpha\in[\alpha_-,\alpha_+]}$ satisfies conditions (X1), (X2) and (X3') in Section \ref{sec proof sec 2}. 
It is well known (see, for example, \cite{Thaler1983, Thaler1995}) that 
$\srbalpha$ is $\sigma$-finite, equivalent to $\lambda$, conservative, ergodic and invariant with respect to $\mapalpha$.  Moreover, for $1\leq\alpha\leq \alpha_+$   
the measure $\srbalpha$ can be written as 
    \begin{align}\label{eq pr formula of phisical}
    \srbalpha=h_\alpha^*(x)\frac{x}{x-f_{1,1}^{-1}(x)} d\lambda(x),
    \end{align}
where $h_\alpha^*$ is continuous and positive on $I$.
Thus, $\{\mapalpha\}_{\alpha\in[\alpha_-,\alpha_+]}$ satisfies conditions (X1) and (X2). (X3') follows from Proposition \ref{prop expansion of tail}.

Next, we shall show that 
\begin{align}\label{eq final}
    \int|\phi-\phi (0)| d\srb<\infty \text{ and }\lim_{\alpha\to 1-0}\int (\phi-\phi(0)) d\srbalpha=\int(\phi-\phi(0)) d\srb.
\end{align}
     By \eqref{eq expansion} and  applying Taylor's theorem to the function $x\mapsto(1+x)^{-1}$ at $x=0$, we obtain 
\[
(f_{1,1}^{-1})'(x)=\frac{1}{f_{1,1}'(f_{1,1}^{-1}(x))}=\frac{1}{1+2b_1x+O(x^{1+\epsilon})}=1-2b_1x+O(x^{1+\epsilon}).
\]
Therefore, we obtain $f_{1,1}^{-1}(x)=x-b_1x^2+O(x^{2+\epsilon})$ and thus, 
\begin{align}\label{eq pr L1}
\frac{x}{x-f_{1,1}^{-1}(x)}=\frac{1}{b_1x+O(x^{1+\epsilon})}.
\end{align}
Since $\phi$ is H\"older continuous, there exist $\eta>0$ and $D\geq 1$ such that for all $x\in I$ we have
\begin{align}\label{eq holder at 0}
    |\phi(x)-\phi(0)|\leq Dx^{\eta}.
\end{align}
Hence, by \eqref{eq pr formula of phisical} and \eqref{eq pr L1}, for all sufficiently small $0<z<1$ we obtain
\begin{align*}
    &\int _{0}^z\frac{|\phi(x)-\phi(0)|h^*_1(x)x}{x-f_{1,1}^{-1}(x)} dx
    \leq D_0\int^z_0\frac{x^\eta}{b_1x+O(x^{1+\epsilon})}dx
    \\&\leq 
    D_0\int^z_0\frac{1}{x^{1-\eta}(b_1+O(x^{\epsilon}))}dx
    \leq \frac{2D_0}{b_1} \int _0^z{x^{-1+\eta}
}dx=\frac{2D_0}{b_1\eta}z^\eta<\infty,
\end{align*}
where $D_0:=D\max_{x\in I}\{h^*_1(x)\}$. Since for all $0<z<1$ the function 
\[
x\mapsto h^*_1(x){x}(x-f_{1,1}^{-1}(x))^{-1}
\]
is continuous on $[z,1]$, we obtain 
\begin{align}\label{eq final L1}
    \int|\phi-\phi (0)| d\srb<\infty.
\end{align}

Since $\indumapalpha$ is  a first return map of $\mapalpha$, kac's formula implies that for each $\alpha \in[\alpha_-,1)$
\begin{align}\label{eq kac alpha}
\int (\phi-\phi(0)) d\srbalpha
=\int \sum_{j=0}^{\returnalpha-1}(\phi-\phi(0))\circ \mapalpha^{j}\cdot h_\alpha d\indulebesgue.
\end{align}
Moreover, by \eqref{eq final L1} and the conservativity of $\srb$, we obtain 
\begin{align}\label{eq kac 1}
    \int (\phi-\phi(0)) d\srb
=\int \sum_{j=0}^{\return_1-1}(\phi-\phi(0))\circ \map_1^{j}\cdot h_1 d\indulebesgue
\end{align}
By the definition of the Ruelle operator $P_\alpha$, for all $\alpha \in[\alpha_-,1]$ we obtain
\begin{align}\label{eq transfer operator}
&\int \sum_{j=0}^{\returnalpha-1}(\phi-\phi(0))\circ \mapalpha^{j}\cdot h_\alpha d\indulebesgue=
\int P_\alpha\left(\sum_{j=0}^{\returnalpha-1}(\phi-\phi(0))\circ \mapalpha^{j}\cdot h_\alpha\right) d\indulebesgue
\\&\nonumber=
\int \sum_{i=2}^m\sum_{k=1}^\infty
G_{\alpha,(i,k)}
\left(
\sum_{j=0}^{\returnalpha-1}(\phi-\phi(0))\circ \mapalpha^{j} \cdot h_\alpha
\right)\circ F^{-1}_{\alpha,(i,k)}
d\indulebesgue
\\&\nonumber=\int \sum_{i=2}^m\sum_{k=1}^\infty
G_{\alpha,(i,k)}
\left(
\sum_{j=0}^{k-1}(\phi-\phi(0))\circ \mapalpha^{j}\circ F^{-1}_{\alpha,(i,k)}
\right)
 h_\alpha
\circ F^{-1}_{\alpha,(i,k)}
d\indulebesgue.
\\&\nonumber=\int 
g(\alpha,x)
d\indulebesgue(x),
\end{align}
where 
\begin{align*}
    &g(\alpha,x):=
    \\&
  \sum_{i=2}^m\sum_{k=1}^\infty
G_{\alpha,(i,k)}(x)
\left(
\sum_{\ell=1}^{k-1}(\phi-\phi(0))\circ f_{\alpha,1}^{-\ell}
+
(\phi-\phi(0))\circ F^{-1}_{\alpha,(i,k)}
\right)\cdot h_\alpha\circ F^{-1}_{\alpha,(i,k)}(x).
\end{align*}
and $\sum_{\ell=1}^{0}(\phi-\phi(0))\circ f_{\alpha,1}^{-\ell}(x):=0$ for all $x\in \indudomain$.
By Proposition \ref{prop expansion of tail}, there exists $D_0>0$ such that for all $\alpha\in [\alpha_-,\alpha_+]$ and $n\in\mathbb{N}$ we have 
\begin{align}\label{eq comparability yan and n}
D_0^{-1}\leq \yan n^{1/\alpha}\leq D_0.
\end{align}
From this estimate and \eqref{eq expansion}, it is not difficult to derive (see, for example, the calculation in the proof of \cite[Lemma 5.3]{Korepanov}) that there exists a constant $D_1\geq1$ such that for all $\alpha\in [\alpha_-,\alpha_+]$, $2\leq k\leq m$ and $k\in\mathbb{N}$ we have 
\begin{align}\label{eq estimate G}
\|G_{\alpha,(i,k)}\|_\infty=\sup_{x\in \indudomain}\left|\frac{d}{dx}\left(f_{\alpha,i}^{-1}\circ f_{\alpha,1}^{-(k-1)}\right)(x)\right|\leq D_1\frac{1}{k^{1+1/\alpha}}.
\end{align}
Furthermore, by the continuity of the map $\alpha\mapsto h_\alpha$ (Teorem \ref{thm korepanov}) and \eqref{eq nonnegative}, there exists $D_2\geq 1$ such that 
\begin{align}\label{eq estimate h}
\sup_{\alpha\in [\alpha_-,\alpha_+]}\|h_\alpha\|_\infty\leq D_2.
\end{align}
By \eqref{eq def y}, \eqref{eq holder at 0} and \eqref{eq comparability yan and n}, for all $2\leq i\leq m$, $k\geq 2$ and $x\in \indudomain$ we obtain 
\begin{align*}
&\left|\sum_{\ell=1}^{k-1}(\phi-\phi(0))(f_{\alpha,1}^{-\ell}(x))\right|
\leq D\sum_{\ell=1}^{k-1}(f_{\alpha,1}^{-\ell}(x))^\eta \leq D\sum_{\ell=1}^{k-1} y_{\alpha,\ell}^\eta
\leq DD_0 \sum_{\ell=1}^{k-1}\frac{1}{\ell^{\eta/\alpha}}
\end{align*}
We take a small number $\delta>0$ with $\delta<\min\{1,\eta/\alpha\}$. Note that, by the integral test, there exists $D_3\geq 1$ such that for all $k\geq 2$ we have $\sum_{\ell=1}^{k-1}\ell^{-\delta}\leq D_3k^{-\delta+1}$.
Thus, for all $2\leq i\leq m$, $k\geq 2$ and $x\in \indudomain$ we obtain 
\begin{align*}
&\left|\sum_{\ell=1}^{k-1}(\phi-\phi(0))(f_{\alpha,1}^{-\ell}(x))\right|
\leq DD_0 \sum_{\ell=1}^{k-1}\frac{1}{\ell^{\eta/\alpha}}\leq DD_0\sum_{\ell=1}^{k-1}\frac{1}{\ell^{\delta}}\leq DD_0D_3k^{-\delta+1}.
\end{align*}
Combining this with \eqref{eq estimate G} and \eqref{eq estimate h}, for all $\alpha\in[\alpha_-,1]$ and $x\in\indudomain$ we obtain 
\begin{align}\label{eq estimate g}
  & |g(\alpha,x)|
\leq
\tilde D
\sum_{i=2}^m\sum_{k=1}^\infty
\frac{1}{k^{1+1/\alpha}}
\left( k^{-\delta+1}
+2\sup_{x\in I}|\phi|\right)
\\&\leq \tilde Dm\max\left\{1,2\sup_{x\in I}|\phi|\right\}\sum_{k=1}^\infty
\frac{1}{k^{1+\delta}}<\infty, \nonumber
\end{align}
where $\tilde D:=DD_0D_1D_2D_3$. 
Since for each $x\in \indudomain$, $2\leq i\leq m$ and $k\in\mathbb{N}$ the maps $\alpha\mapsto G_{\alpha,(i,k)}(x)$,
$\alpha\mapsto f_{\alpha,1}^{-1}(x)$,
$\alpha\mapsto F^{-1}_{\alpha,(i,k)}(x)$ and $\alpha \mapsto h_\alpha(x)$ are continuous on $[\alpha_-,\alpha_+]$ (see (f5) and Theorem \ref{thm korepanov}), the function $\alpha\mapsto g(\alpha,x)$ is continuous on $[\alpha_-,1]$ for each $x\in \indudomain$. 
Therefore, by \eqref{eq estimate g} and the Lebesgue dominated convergence theorem, we obtain
$
\lim_{\alpha\to 1-0}\int g(\alpha,x) d\indulebesgue(x)=\int g(1,x) d\indulebesgue(x).
$
Combining this with \eqref{eq kac alpha}, \eqref{eq kac 1} and \eqref{eq transfer operator}, we obtain 
\begin{align*}
    \lim_{\alpha\to 1-0}\int (\phi-\phi (0))d\srbalpha=\int (\phi -\phi(0))d\srb.
\end{align*}
This completes the proof of \eqref{eq final}. Thus, the proof of Theorem \ref{thm main abstract} is complete.
\qed


\section{Linear response for Pomeau–Manneville maps at the transition point}\label{sec PMmap}
In this section, we show that for Example \ref{example MP map} linear response holds for the SRB measure at the transition point but fails for the physical measure. Throughout this section, we use the notation introduced in Example \ref{example MP map} and Section \ref{sec example}.

For $\alpha>0$ we define the map $\lsvmapalpha:I\rightarrow{I}$ by \eqref{eq def lsv map intro}. Let $0<\alpha_-<\alpha_+$ and let $\alpha\in[\alpha_-,\alpha_+]$. Since $\indumapalpha:[1/2,1]$ is the first return map of $\lsvmapalpha$, we have 
\[
\srbalpha|_{[1/2,1]}:=\densityalpha d\lambda|_{[1/2,1]}=\indudensityalpha d\indulebesgue=:\indusrbalpha,  \text{where } \indulebesgue=2\lambda|_{[1/2,1]}
\]
(see \cite[Corollary 1.4.4]{viana}). Thus, since $\densityalpha$ is the Radon-Nikodym derivative of $\srbalpha$ with respect to $\lambda$, chosen to be continuous on $(0,1]$, we obtain $\densityalpha=2\indudensityalpha$. 
Therefore, by Bahsoun and Saussol \cite{BahsounSaussol}, Remark \ref{rem verify LSv map}, and Theorem \ref{thm main abstract}, we obtain the following:
\begin{thm}\label{thm main pomeau}
Let $\alpha_-<1$ and let $\alpha_+>1$. Let $\{\lsvmapalpha\}_{\alpha\in [\alpha_-,\alpha_+]}$ be the family of maps such that $\lsvmapalpha$ defined by \eqref{eq def lsv map intro}. Then, we have the following:
\begin{itemize}
    \item For each H\"older continuous potential $\psi$ on $I$ with $\psi(0)=0$ the map
\[
\alpha\mapsto R_{\text{SRB},\psi}(\alpha):=\int \psi d\srbalpha
\]
is differentiable at the transition point $1$.
\item For each H\"older continuous potential $\phi$ on $I$ the map
\[
\alpha\mapsto R_{\text{Phy},\phi}(\alpha):=\int \phi d\phisicalalpha
\]
is differentiable at the transition point $1$ if and only if 
\[
\int (\phi-\phi(0))d\nu_{1}=0.
\]
\end{itemize}
In addition, we have
\[
\lim_{\alpha\to1-0}\frac{R_{\text{Phy},\phi}(\alpha)-R_{\text{Phy},\phi}(1)}{\alpha-1}
=-\frac{8}{\density(1/2)}\int (\phi-\phi(0))d\nu_{1}.
\]
and 
\begin{align*}
\lim_{\alpha\to1-0}
\operatorname*{ess\,lim}_{n\to\infty}
\left|
\frac{\frac{1}{n}\sum_{i=0}^{n-1} \phi\circ\lsvmapalpha^i-\phi(0)}{\alpha-1}-\left(-\frac{8}{\density(1/2)}\int (\phi-\phi(0))d\srb\right)\right|=0.
\end{align*}

\end{thm}

\subsection*{Acknowledgments}
This work was supported by the JSPS KAKENHI 25KJ1382.

\subsection*{Data Availability}
No datasets were generated or analysed during the current study.

\subsection*{Statements and Declarations}
{Competing Interests:} The author declares that there are no competing interests.

\bibliographystyle{abbrv}
\bibliography{reference}
 \nocite{*}
\end{document}